\documentclass{amsart}

\usepackage{color,amsmath,amssymb,setspace,savesym,hyperref}
\usepackage[all]{xy}
\savesymbol{cir}
\usepackage[shortalphabetic]{amsrefs}

\theoremstyle{plain}
\newtheorem{theorem}{Theorem}[section]
\newtheorem*{theorem*}{Theorem}
\newtheorem{lemma}[theorem]{Lemma}
\newtheorem{proposition}[theorem]{Proposition}
\newtheorem{corollary}[theorem]{Corollary}
\newtheorem*{corollary*}{Corollary}

\newtheorem*{conjecture*}{Conjecture}

\theoremstyle{remark}
\newtheorem{remark}[theorem]{Remark}

\theoremstyle{definition}

\newcommand{\Q}{\mathbb{Q}}

\newcommand{\N}{\mathbb{N}}
\newcommand{\dis}{\displaystyle}

\newcommand{\de}{\delta}
\newcommand{\ep}{\epsilon}

\newcommand{\G}{\Gamma}
\newcommand{\D}{\Delta}
\newcommand{\La}{\Lambda}
\newcommand{\Si}{\Sigma}

\newcommand{\I}{^{-1}}

\DeclareMathOperator{\Aut}{Aut}
\DeclareMathOperator{\Out}{Out}
\DeclareMathOperator{\lk}{lk}
\DeclareMathOperator{\st}{st}
\DeclareMathOperator{\dlk}{{\lk}{\downarrow}}
\DeclareMathOperator{\dst}{{\st}{\downarrow}}

\newcommand{\SAut}{\Si\!\Aut}  

\numberwithin{equation}{section}

\begin{document}

\title[Stability for symmetric automorphism groups]{Rational homological stability for groups of symmetric automorphisms of free groups}

\author{Matthew C. B. Zaremsky}


\address{Fakult\"at f\"ur Mathematik \\
Universit\"at Bielefeld \\
Bielefeld, Germany 33615}
\email{zaremsky@math.uni-bielefeld.de}

\begin{abstract}
\singlespacing
Let $F_n$ be the free group of rank $n$, with generating set $S=\{x_1,\dots,x_n\}$. An automorphism $\phi$ of $F_n$ is called symmetric if for each $1\leq i\leq n$, $\phi(x_i)$ is conjugate to $x_j$ or $x_j\I$ for some $1\leq j\leq n$. Let $\SAut(F_n)$ be the group of symmetric automorphisms. We prove that the inclusion $\SAut(F_n)\rightarrow\SAut(F_{n+1})$ induces an isomorphism in rational homology for $n>(3i-1)/2$.
\end{abstract}

\maketitle

\section{Introduction}
\label{sec:intro}

Let $\Aut(F_n)$ be the group of automorphisms of the free group $F_n$. If $S=\{x_1,\dots,x_n\}$ is a fixed basis of $F_n$, an automorphism $\phi$ of $F_n$ is called \emph{symmetric} if for each $1\leq i\leq n$, $\phi(x_i)$ is conjugate to $x_j$ or $x_j\I$ for some $1\leq j\leq n$. If $\phi$ is an automorphism such that each $\phi(x_i)$ is even conjugate to $x_i$ we call $\phi$ \emph{pure symmetric}. Let $\SAut(F_n)$ be the group of symmetric automorphisms of $F_n$ and $P\SAut(F_n)$ the group of pure symmetric automorphisms.

In \cite{hv98} it is shown that the homology of $\Aut(F_n)$ is stable with respect to $n$, and in \cite{galatius11} the stable rational homology is even shown to be trivial. Namely, $H_i(Aut(F_n);\Q)=0$ for all $n>2i+1$ \cite{galatius11}*{Corollary~1.2}. For the pure symmetric case, in \cite{jensen_wahl04} it is shown that the rational homology of $P\SAut(F_n)$ is in fact \emph{not} stable. The symmetric case $\SAut(F_n)$ is the situation of present interest. In \cite{hw10} it is shown that $\SAut(F_n)$ is indeed homologically stable, as a corollary to a much larger project. In particular, the inclusion $\SAut(F_n)\rightarrow\SAut(F_{n+1})$ induces an isomorphism in homology for $n>2i+1$, and a surjection for $n=2i+1$ \cite{hw10}*{Corollary~1.2}. In this note we prove rational homological stability using some simple arguments from combinatorial Morse theory, inspired by methods used in \cite{bm09}.

\begin{theorem}\label{hom_stab_thm}
 The map $H_i(\SAut(F_n);\Q)\rightarrow H_i(\SAut(F_{n+1});\Q)$ induced by inclusion is an isomorphism for $n>(3i-1)/2$.
\end{theorem}

J. Griffin and J. Wilson have brought to our attention the independent papers \cite{griffin12} and \cite{wilson11}, in which a much stronger result is proved using an entirely different method, with Theorem~\ref{hom_stab_thm} as a trivial consequence, namely

\begin{theorem}\citelist{\cite{griffin12} \cite{wilson11}}\label{trivial_hlgy}
 $\widetilde{H}_i(\SAut(F_n);\Q)=0$ for all $n$ and $i$.
\end{theorem}

The present work is thus not a new result, but the method developed here could have further applications to understanding the geometry of these groups. In particular Theorem~\ref{trivial_hlgy} implies that the space $\Si Q_n$ defined in Section~\ref{sec:auterspace} is $\Q$-acyclic, and our method used here strongly suggests that it should even be contractible. Also, the present method has recently proved useful in analyzing the \emph{partially symmetric} case, in \cite{zar12}, which essentially subsumes the present work.

In Section~\ref{sec:auterspace} we provide some background on the spines of Auter space $K_n$ and ``cactus space'' $\Si K_n$, and present a convenient stratification of $\Si K_n$ into sublevel sets $\Si K_{n,c}$. Then in Section~\ref{sec:morse} we introduce a Morse function refining this stratification, and show that each $\Si K_{n,c}$ is $c-1$ connected. Lastly, in Section~\ref{sec:hom_stab} we prove Theorem~\ref{hom_stab_thm}.

\subsection*{Acknowledgements} The author is grateful to Kai-Uwe Bux and Rob McEwen for many helpful conversations, and to James Griffin and Jenny Wilson for alerting him to the results from \cite{griffin12} and \cite{wilson11}.

\section{Auter space and cactus space}\label{sec:auterspace}

We will analyze the homology of $\SAut(F_n)$ by considering its action on a certain topological space. Our starting point is the well-studied \emph{spine of Auter space} $K_n$ introduced in \cite{hv98}. Let $R_n$ be the rose with $n$ edges, i.e., the graph with a single vertex $p_0$ and $n$ edges. Here by a \emph{graph} we mean a connected one-dimensional CW-complex, with the usual notions of vertices and (oriented) edges. We identify $F_n$ with $\pi_1(R_n)$. If $\G$ is a graph with basepoint vertex $p$, a homotopy equivalence $\rho:R_n\rightarrow\G$ is called a \emph{marking} on $\G$ if $\rho$ takes $p_0$ to $p$. We will consider two markings to be equivalent if there is a basepoint-preserving homotopy between them. Also, we only consider graphs that are \emph{reduced}, i.e., $p$ is at least 2-valent, all other vertices are at least 3-valent, and $\G$ has no separating edges; see \citelist{\cite{collins89} \cite{hv98}}. The spine $K_n$ of Auter space is then the space of (reduced) marked basepointed graphs $(\G,p,\rho)$, up to equivalence of markings, such that $\pi_1(\G)=F_n$.

To be more precise, $K_n$ is a simplicial complex, with a vertex for every equivalence class $(\G,p,\rho)$. An $r$-simplex with vertices $\G_0,\G_1,\dots,\G_r$ is a chain of \emph{forest collapses} $\G_r\stackrel{d_r}{\rightarrow}\G_{r-1}\stackrel{d_{r-1}}{\rightarrow}\cdots\stackrel{d_1}{\rightarrow}\G_0$ such that there exist markings $\rho_i:R_n\rightarrow\G_i$ with the following diagram commuting up to homotopy.

$$\xymatrix{
\G_r \ar[r]^{d_r}
& \G_{r-1} \ar[r]^{d_{r-1}}
& \cdots \ar[r]^{d_2} & \G_1 \ar[r]^{d_1} & \G_0 \\
& & R_n \ar[ull]^{\rho_r} \ar[ul]_{\rho_{r-1}} \ar[ur]^{\rho_1} \ar[urr]_{\rho_0}
}$$
Here a \emph{forest collapse} or \emph{blow-down} $d:\G\rightarrow\G'$ is a (basepoint-preserving) homotopy equivalence of graphs that amounts to collapsing a subforest of $\G$. The reverse of a blow-down is, naturally, called a \emph{blow-up}.

Since we are identifying $F_n$ with $\pi_1(R_n)$, we can also identify $\Aut(F_n)$ with the group of basepoint-preserving homotopy equivalences of $R_n$, up to homotopy. This is of course the same as the group of markings of $R_n$, so we can denote markings on $R_n$ by elements of $\Aut(F_n)$. There is a (right) action of $\Aut(F_n)$ on $K_n$ in the following way: given $(\G,p,\rho)\in K_n$ and $\phi\in\Aut(F_n)$, we have $\phi(\G,p,\rho)=(\G,p,\rho\circ\phi)$. In particular this action only affects markings, and it is easy to see that $\Aut(F_n)$ permutes markings arbitrarily.

To analyze the groups $\SAut(F_n)$ we will work with a certain subcomplex (\emph{``cactus space''}) $\Si K_n$ of $K_n$. The vertices of $\Si K_n$ are marked basepointed graphs $(\G,p,\rho)$ such that $\G$ is a \emph{cactus graph} and $\rho$ is a \emph{symmetric marking}. A cactus graph is a graph $\G$ such that every edge is contained in a unique reduced cycle. A marking $\rho$ is called symmetric if for any maximal tree $T$ in $\G$ with $\pi:\G\rightarrow\G/T=R_n$ we have $\pi\circ\rho\in \SAut(F_n)$ (recall our identification of $\Aut(F_n)$ with markings on $R_n$). For brevity we will just define a \emph{marked cactus graph} to be a cactus graph with a symmetric marking. Since any blow-down of a cactus graph is again a cactus graph, any simplex spanned by vertices in $\Si K_n$ is itself in $\Si K_n$, so this really is a full subcomplex of $K_n$. See \cite{collins89} for a more complete discussion of cactus graphs and symmetric markings. Also see \cite{bcv09} for a generalization of this to the partially symmetric case.

The action of $\SAut(F_n)$ on $K_n$ only affects markings, and takes symmetric markings to symmetric markings, so we can consider the action of $\SAut(F_n)$ on $\Si K_n$. Let $\Si Q_n:=\Si K_n/\SAut(F_n)$ be the orbit space. This setup is analyzed in \cite{collins89}, with the following important results.

\begin{proposition}\cite{collins89}*{Proposition~3.5, Theorem~4.7}\label{collins_prop}
 $\SAut(F_n)$ acts on $\Si K_n$ with finite stabilizers and finite quotient $\Si Q_n$, and $\Si K_n$ is contractible.
\end{proposition}

It is also clear that if an element of $\SAut(F_n)$ stabilizes a simplex then it fixes it pointwise, since the vertices of any simplex correspond to pairwise non-isomorphic graphs. The upshot of this that $\Si Q_n$ and $\SAut(F_n)$ have the same rational homology; see for example Exercise~2 on page 174 in \cite{brown82}.

It is difficult to analyze $\Si Q_n$ directly, and so we will work with a certain stratification, similar to the one used in \cite{hv98} and \cite{bm09} for $Q_n$. For a cactus graph $\G$ with basepoint $p$, define the \emph{weight} $b(\G)$ to be the number of reduced cycles in $\G$ that contain the basepoint $p$ as a vertex. Also define the \emph{coweight} $c(\G)$ to be $n-b(\G)$, i.e., the number of reduced cycles that do not contain $p$. Obviously $0\leq c(\G)\leq n-1$. For $c\in\N_0$ let $\Si K_{n,c}$ be the subspace of $\Si K_n$ spanned by marked basepointed graphs of coweight at most $c$. Note that $\Si K_{n,0}$ is the discrete set of marked roses, and $\Si K_{n,c}=\Si K_n$ if and only if $c\geq n-1$.

\begin{remark}
 Our definition of coweight is related to the \emph{degree} of a graph, used in \cite{hv98} and \cite{bm09}. One can check that coweight is really just half the degree, and so is essentially the same measurement, but in the cactus setting the notion of coweight is more natural.
\end{remark}

\section{Morse theory}\label{sec:morse}

We now define a height function $h$ on $\Si K_n$. This height function is related to the one defined in \cite{bm09} on the space $K_n$, though is adjusted to work nicely with cactus graphs. Let $(\G,p,\rho)$ be a basepointed marked cactus graph. For vertices $v,v'$ in $\G$, define the distance $d(v,v')$ to be the number of reduced cycles in $\G$ that share an edge with a minimal-length path from $v$ to $v'$. Since $\G$ is a cactus graph, it is easy to see that this is well-defined. Let $\La_i(\G):=\{v\in\G\mid d(p,v)=i\}$, so for example $\La_0(\G)=\{p\}$. For each $i\geq0$ define $n_i(\G):=-|\La_i(\G)|$. Note that for a given reduced cycle $C$, there is a unique vertex $v$ of $C$ with minimal distance to $p$. Call this vertex the \emph{base} of $C$. Define $c_i(\G)$ to be the number of reduced cycles in $\G$ whose base is not in $\La_i$. This is a sort of ``local coweight,'' and in particular when $i=0$ we recover the coweight, $c=c_0$. Define $h$ to be
$$h(\G)=(c_0(\G),n_1(\G),c_1(\G),n_2(\G),c_2(\G)\dots)$$
with the lexicographic order. Note that we should technically write $h(\G,p,\rho)$, but since $h$ is independent of $\rho$, and $p$ is understood, we will usually just write $h(\G)$. This is a refinement of $c$, and we will use $h$ to analyze the connectedness of $\Si K_{n,c}$.

For any vertex $(\G,p,\rho)$ define $\Si K_n^{\leq\G}$ to be the subcomplex of $\Si K_n$ having vertices $(\G',p',\rho')$ with $h(\G')\leq h(\G)$. We have natural notions of the \emph{descending star} and \emph{descending link} with respect to $h$. For a vertex $\G$ in $\Si K_n$, the descending star $\dst(\G)$ with respect to $h$ is the open star of $\G$ in $\Si K_n^{\leq\G}$, that is the set of simplices in $\Si K_n^{\leq\G}$ that contain $\G$. The descending link $\dlk(\G)$ is the set of faces of simplices in $\dst(\G)$ that do not themselves contain $\G$. The descending link can be described as the join of the descending blow-down complex, or \emph{down-link}, and the descending blow-up complex, or \emph{up-link}; see \cite{bm09}. Here we say that a blow-down or blow-up is descending if the resulting graph has a lower height than the starting graph. Of course in the up-link we restrict to blow-ups that result in cactus graphs.

For an edge $e$ in $\G$ with vertices $v$ and $v'$, we call $e$ \emph{horizontal} if $d(p,v)=d(p,v')$. Otherwise we call $e$ \emph{vertical}. If a graph $\G$ has the property that every reduced cycle has length at most 2, we will call $\G$ \emph{thin}. Otherwise we call $\G$ non-thin. Note that a graph is thin if and only if every horizontal edge is a loop. For a forest $F$ in $\G$, let $D(F)$ be the smallest $i$ such that $F$ has a vertex in $\La_i$.

\begin{lemma}\label{proper_morse}
 If $F$ connects two vertices in $\La_{D(F)}$ then $h(\G/F)>h(\G)$. If $F$ does not connect any vertices in $\La_{D(F)}$, then $h(\G/F)<h(\G)$.
\end{lemma}

\begin{proof}
 First suppose $F$ connects two vertices in $\La_{D(F)}$. Then $n_{D(F)}(\G/F)>n_{D(F)}(\G)$, and none of the $n_i$ or $c_i$ change for $i<D(F)$, so $h(\G/F)>h(\G)$. Now suppose $F$ does not connect any vertices in $\La_{D(F)}$. Then $n_{D(F)}(\G/F)=n_{D(F)}(\G)$, and again none of the $n_i$ or $c_i$ change for $i<D(F)$. However, since $\G$ is reduced, blowing down $F$ must increase the number of reduced cycles with base in $\La_{D(F)}$. Thus $c_{D(F)}(\G/F)<c_{D(F)}(\G)$, and so $h(\G/F)<h(\G)$.
\end{proof}

This tells us that adjacent vertices in $\Si K_n$ have different heights, and so descending stars of vertices with the same height are disjoint. It also implies that a single edge is a descending forest if and only if it is vertical. We now analyze the descending links of vertices in $\Si K_n$.

\begin{lemma}\label{downlink}
 If $\G$ is non-thin, the down-link of $\G$ is contractible. If $\G$ is thin, the down-link of $\G$ is homotopy equivalent to a $V-2$ sphere, where $V$ is the number of vertices of $\G$.
\end{lemma}

\begin{proof}
 The down-link of $\G$ is realized by the poset $P(\G)$ of descending forests in $\G$. First suppose $\G$ is non-thin. Let $C$ be a reduced cycle of length at least 3. Then $C$ has precisely two vertical edges and at least one horizontal edge. Let $\ep$ be a vertical edge and $\eta$ a horizontal edge. Then the forest consisting just of $\ep$ is a descending forest, and the forest consisting of $\eta$ is not. So, apply the map $F\mapsto F\setminus\eta$ to $P(\G)$. This is clearly a poset retraction onto the poset $P_0(\G)$ of descending forests that do not contain $\eta$. Since $F\setminus\eta\subseteq F$ for all $F$, by \cite{quillen78}*{1.3} this induces a homotopy equivalence between $P(\G)$ and $P_0(\G)$. Now apply the map $F\mapsto F\cup\ep$ to $P_0(\G)$. Since $F$ does not contain $\eta$, $F\cup\ep$ is a forest, and is still descending since $\ep$ is. This then induces a homotopy equivalence between $P_0(\G)$ and the star of $\ep$, which is visibly contractible, with cone point $\ep$.

 Now suppose $\G$ is thin. Then every subforest of $\G$ is a descending subforest, and so $P(\G)$ is identical to $F(\G)$, the complex of all forests in $\G$. Moreover, since every edge of $\G$ is contained in a unique reduced cycle, we have that $F(\G)$ is simply the join of the $F(C)$, as $C$ ranges over every reduced cycle. But clearly $F(C)$ is either empty, if $C$ is a loop, or is $S^0$, if $C$ has two edges. Thus $F(\G)$ is homotopy equivalent to a sphere. By \cite{vogt90}*{Proposition~2.2}, $F(\G)$ is homotopy equivalent to a wedge of spheres of dimension $V-2$, and so it actually must be a single sphere of dimension $V-2$.
\end{proof}

For non-thin $\G$, we now know that $\dlk(\G)$ is contractible, as the join of the contractible down-link with the up-link. Now suppose $\G$ is thin, so the down-link is $S^{V-2}$, and consider the up-link. For any vertex $v\neq p$, there is a unique reduced cycle $C_v$ containing $v$ such that $v$ is not the base of $C_v$. If a blow-up $B$ at a vertex $v\neq p$ separates the two half-edges of $C_v$ incident to $v$, we call $B$ \emph{separating} at $v$. Let $\Si BU(v)$ denote the poset of separating blow-ups at $v$. Let $b(v)$ denote the number of reduced cycles whose base is $v$ (the ``weight at $v$'').

\begin{lemma}\label{local_blowups_conn}
 For $v\neq p$, $\Si SBU(v)\simeq S^{b(v)-2}$.
\end{lemma}

\begin{proof}
 As in Section~2 of \cite{bm09}, we will use the combinatorial framework for blow-ups described in \cite{cv86}, as the poset of compatible partitions. Label the half-edges incident to $v$ that are not in $C_v$ with $a_1,\bar{a}_1,\dots,a_{b(v)},\bar{a}_{b(v)}$. Also label the half-edges incident to $v$ that are in $C_v$ with $a_0$ and $\bar{a}_0$. Do this so that each $a_i$ shares a reduced cycle precisely with $\bar{a}_i$. We now consider partitions of $\{a_0,\bar{a}_0,a_1,\bar{a}_1,\dots,a_{b(v)},\bar{a}_{b(v)}\}$ into two blocks $\{a,\bar{a}\}$ such that the size of each block is at least two. Since we only consider cactus blow-ups, we only consider partitions such that for precisely one $0\leq i\leq b(v)$, $a_i$ and $\bar{a}_i$ do not share a block. The partition $\{a,\bar{a}\}$ is separating at $v$ if and only if $a_0$ and $\bar{a}_0$ do not share a block. We say two partitions $\{a,\bar{a}\}$ and $\{b,\bar{b}\}$ are compatible if either $a\subseteq b$ or $b\subseteq a$. Let $S$ denote the simplicial complex of partitions, where the vertices are the partitions and each collection of $r+1$ pairwise compatible partitions spans an $r$-simplex. Let $S_0$ be the subcomplex spanned by separating partitions, and $S_1$ the star of $S_0$ in $S$. Then $S_1$ is clearly homotopy equivalent to $\Si SBU(v)$, and thus so is $S_0$. But $S_0$ is the surface of the barycentric subdivision of a $(b(v)-1)$-simplex with vertices labeled $1,\dots,b(v)$, so indeed $S_0\simeq S^{b(v)-2}$.
\end{proof}

\begin{corollary}\label{global_blowups_conn}
 Let $\dis \Si SBU(\G):=\ast_{v\in\G-p}\Si SBU(v)$. Let $c$ be the coweight of $\G$ and $V$ the number of vertices. Then $\Si SBU(\G)\simeq S^{c-V}$.
\end{corollary}

\begin{proof}
 We know by Lemma~\ref{local_blowups_conn} that
 $$\Si SBU(\G) \simeq \ast_{v\in\G-p}S^{b(v)-2},$$
 which is homotopic to a sphere of dimension $\dis ((V-2)+\sum_{v\neq p}(b(v)-2))$. But $\dis (V-2)+\sum_{v\neq p}(b(v)-2)=(v-2)+c-2(V-1)=c-V$.
\end{proof}

\begin{lemma}\label{uplink_conn}
 For thin $\G$ having $V$ vertices and coweight $c$, the up-link of $(\G,p,\rho)$ in $\Si K_n$ is homotopy equivalent to $S^{c-V}$.
\end{lemma}

The proof is essentially the same as the proof of Lemma~2.5 in \cite{bm09}.

\begin{proof}
 We claim that the up-link is homotopy equivalent to $\Si SBU(\G)$. For a poset $P$, let $\underline{P}$ be $P\sqcup\{\perp\}$, where $\perp$ is a formal minimal element. Then we have that $P\ast Q\simeq \underline{P}\times\underline{Q} - \{(\perp,\perp)\}$. Let
 $$X:=\{f\in\prod_{v\neq p}\underline{\Si BU}(v)\mid \exists v\in\La_{D(f)} \textnormal{ with } f_v\in\Si SBU(v)\}.$$
 Here $f_v$ is a blow-up at $v$ in the tuple $f$, and $D(f)$ denotes the smallest $i$ such that $f_v\neq\perp$ for some $v\in\La_i$. Define a map $r:X\rightarrow\Si SBU(\G)$ via
 \begin{align*}
  (f_v)_{v\neq p}\mapsto\left(\left\{
  \begin{matrix}
   f_v & \textnormal{ for } & f_v\in\Si SBU(v)\\
   \perp & \textnormal{ for } & f_v\not\in\Si SBU(v)
  \end{matrix}
  \right.\right)_{v\neq p}
 \end{align*}
 This is just a restriction of the map $r$ used in the proof of \cite{bm09}*{Lemma~2.5}, and it is clearly a poset retraction onto $\Si SBU(\G)$. Also, $r(f)\leq f$ for all $f\in X$, and so by \cite{quillen78}*{1.3} this induces a homotopy equivalence between the geometric realization of $X$ and $\Si SBU(\G)$. But the geometric realization of $X$ is the up-link of $\G$, and so by Lemma~\ref{global_blowups_conn}, the up-link is homotopy equivalent to $S^{c-V}$.
\end{proof}

\begin{proposition}\label{dlk_conn}
 For any cactus graph $\G$ with $c=c(\G)$, either $\dlk(\G)$ is contractible or $\dlk(\G)\simeq S^{c-1}$.
\end{proposition}

\begin{proof}
 If $\G$ is non-thin, $\dlk(\G)$ is contractible. If $\G$ is thin, the down-link is homotopy equivalent to $S^{V-2}$ and the up-link is homotopy equivalent to $S^{c-V}$. This implies that the descending link is homotopy equivalent to $S^{c-1}$.
\end{proof}

Having shown that the descending links are highly connected, basic Morse theory tells us that the sublevel sets are also highly connected, since $\Si K_n$ is contractible. Also see the Theorem in \cite{bm09}.

\begin{corollary}\label{auter_cac_space_conn}
 $\Si K_{n,c}$ is $c-1$ connected.
\end{corollary}

\begin{proof}
 First note that $\Si K_{n,n-1}=\Si K_n$ is contractible, and thus $n-2$ connected. Assuming $\Si K_{n,c+1}$ is $c$-connected, we need to show that $\Si K_{n,c}$ is $c-1$ connected. But $\Si K_{n,c+1}$ is obtained from $\Si K_{n,c}$ by gluing on the descending stars of graphs with coweight $c+1$ along their descending links, and all these descending links are $c-1$ connected by Proposition~\ref{dlk_conn}, so this follows immediately.
\end{proof}

\section{Homological stability}\label{sec:hom_stab}

Everything in this section is basically just the version of Section~5 in \cite{hv98} for the symmetric case, and no major changes are necessary. The only real difference is that we have used coweight as our coarse stratification, rather than degree, but since $\Si K_{n,c}$ is $c-1$ connected, the numbers all work out exactly as in \cite{hv98}.

Consider the action of $\SAut(F_n)$ on $\Si K_n$. For each $c$, $\Si K_{n,c}$ is clearly stabilized by $\SAut(F_n)$, so it makes sense to define $\Si Q_{n,c}:=\Si K_{n,c}/\SAut(F_n)$. As explained in \cite{hv98}, since $\Si K_n$ is contractible and each $\Si K_{n,c}$ is $c-1$ connected, it is easy to see that $\Si Q_{n,c}$ has the same rational homology as $\SAut(F_n)$ in dimensions less than $c$; see for example Exercise~2 on page 174 of \cite{brown82}. To be precise, we have the following

\begin{lemma}\label{gp_to_space_hlgy}
 $H_i(\Si Q_{n,c};\Q)\cong H_i(\SAut(F_n);\Q)$ for all $i<c$. Also, $H_c(\Si Q_{n,c};\Q)$ surjects onto $H_c(\SAut(F_n);\Q)$.\qed
\end{lemma}

To get homological stability for $\SAut(F_n)$ we can now look for homological stability of $\Si Q_{n,c}$. Note that the vertices of $\Si Q_n$ are just the homeomorphism types of basepointed graphs, since $\SAut(F_n)$ changes the markings arbitrarily. Embed $\Si K_{n,c}$ into $\Si K_{n+1,c}$ by sending $(\G,p,\rho)$ to $(\G\vee S^1,p,\rho')$. Here $\rho'$ is $\rho$ extended to $\G\vee S^1$ by sending the new generator to the new loop. This is the same embedding as described in \cite{hv98} for the $K_n$ case. This induces an embedding $\iota:\Si Q_{n,c}\hookrightarrow\Si Q_{n+1,c}$.

We now describe a way to ``detect'' the presence of certain subgraphs at the basepoint. If $\G$ has rank $n+1$ and has a loop at the basepoint $p$ then $\G$ is in the image of $\iota$, so we want to be able to detect loops. We also want to be able to detect a \emph{loop-digon pair}. This is a subgraph $\de$ consisting of three edges and two vertices, with two of the edges forming a digon and the third edge forming a loop at one vertex. If the other vertex is $p$, we call $\de$ a loop-digon pair \emph{at} $p$. For reduced cycles $C,C'$ we will say that $C'$ is \emph{above} $C$ if the base of $C'$ is a vertex of $C$ that is not the base of $C$.

\begin{lemma}\label{detect_things_at_p}
 Let $(\G,p)$ be a basepointed cactus graph with weight $b$, coweight $c$, and rank $n$. If $2c<n$ then $\G$ has a loop at $p$. If $c<2n/3$ then either $\G$ has a loop or a loop-digon pair at $p$.
\end{lemma}

\begin{proof}
 First suppose $\G$ has no loops at $p$. Then every reduced cycle with base $p$ has at least one reduced cycle above it, so $b\leq c$. But $n=b+c$, so $n\leq 2c$. Now suppose $\G$ has no loops or loop-digon pairs at $p$, so every reduced cycle with base $p$ has at least two reduced cycles above it. Then $b\leq c/2$, and since $n=b+c$ we conclude that $n\leq 3c/2$.
\end{proof}

\begin{remark}
 In \cite{hv98}, a similar lemma detects loops and theta subgraphs at $p$. Cactus graphs have no theta subgraphs, but in the cactus case these loop-digon pairs turn out to yield essentially the same result.
\end{remark}

\begin{proposition}\label{orbit_spaces_stable}
 The map $\iota:\Si Q_{n,c}\hookrightarrow\Si Q_{n+1,c}$ is a homeomorphism for $2c<n+1$ and a homotopy equivalence for $3c/2<n+1$.
\end{proposition}

The proof very closely mirrors the proof of Proposition~5.4 in \cite{hv98}.

\begin{proof}
 If $2c<n+1$ then every $\G$ in $\Si Q_{n+1,c}$ has a loop at $p$, so $\iota$ is a homeomorphism. Now suppose $3c/2<n+1$, and let $\G$ be a vertex not in the image of $\iota$. Then $\G$ has no loops at $p$ but does have at least one loop-digon pair. Let $\D$ be the subgraph of $\G$ consisting of all loop-digon pairs at $p$, say there are $r\geq 1$ of them. Then $\G=\D\vee\G'$, for some $\G'$ with rank $n+1-2r$. Now, the open star of $\G$ in $Q_{n+1,c}$ is the product of open stars of $\D$ in $Q_{2r,r}$ and $\G'$ in $Q_{n+1-2r,c-r}$. The former consists of a single simplex, since all non-loop edges in $\D$ are equivalent under automorphisms of $\D$; moreover, every other vertex of this star has lower coweight since blowing down any edge reduces $c$ by 1. So, collapsing any non-loop edge of $\D$ gives a deformation retraction of the star of $\G$ into the image of $\iota$.
\end{proof}

Since $\iota$ is natural with respect to $\SAut(F_n)\hookrightarrow \SAut(F_{n+1})$, we can now prove our main result.

\begin{proof}[Proof of Theorem~\ref{hom_stab_thm}]
 We know that when $3c/2<n+1$,
 $$H_i(\SAut(F_n);\Q)\rightarrow H_i(\SAut(F_{n+1});\Q)$$
 is an isomorphism for all $i<c$, by Lemma~\ref{gp_to_space_hlgy} and Proposition~\ref{orbit_spaces_stable}. If $n>(3i-1)/2$ then we can take $c=i+1$ and get that $3c/2=3(i+1)/2<n+1$, and so the result follows.
\end{proof}

\end{document}